\documentclass[12pt]{article}
\usepackage{amsmath}
\usepackage{graphicx,psfrag,epsf}
\usepackage{enumerate}
\usepackage{natbib}
\usepackage{url} 

\usepackage{standalone}
\usepackage{amsthm}
\usepackage{yuyuan, braket, enumitem}
\usepackage{mathtools}
\mathtoolsset{showonlyrefs}
\newcommand{\hc}[1]{\hat{\vc{#1}}}
\def\pto{\stackrel{P}{\to}}

\newtheorem{assum}{Assumption}

\newcommand{\Appendix}{\appendix\def\thesection{Appendix~\Alph{section}}
	\def\thesubsection{\Alph{section}.\arabic{subsection}}
}
\makeatletter \renewcommand{\theequation}{\thesection.\arabic{equation}} \@addtoreset{equation}{section} \makeatother
\renewcommand{\vc}[1]{\ensuremath{\boldsymbol{#1}}}

\usepackage{pdflscape,multirow,array} 
\usepackage{tikz}
\newcommand\independent{\protect\mathpalette{\protect\independenT}{\perp}}
\def\independenT#1#2{\mathrel{\rlap{$#1#2$}\mkern2mu{#1#2}}}

\usepackage{amssymb} 
\newtheorem{theorem}{Theorem}

\usepackage{color} 

\newcommand{\blind}{1}

\begin{document}
	
	\vskip 3mm




\if1\blind
{
  {
  \noindent AN ASYMPTOTIC RESULT OF CONDITIONAL LOGISTIC REGRESSION ESTIMATOR
   }
} \fi

\if0\blind
{
  \bigskip
  \bigskip
  \bigskip
  \begin{center}
    {
	  \noindent AN ASYMPTOTIC RESULT OF CONDITIONAL LOGISTIC REGRESSION ESTIMATOR
    }
\end{center}
  \medskip
} \fi

\vskip 3mm

\vskip 5mm
\noindent Zhulin He and Yuyuan Ouyang

\noindent Center for Survey Statistics and Methodology and Department of Statistics

\noindent Iowa State University

\noindent Ames, Iowa \ 50011

\noindent zhulin.stat@gmail.com

\noindent School of Mathematical and Statistical Sciences

\noindent Clemson University

\noindent Clemson, South Carolina \ 29634

\noindent yuyuano@clemson.edu

\vskip 3mm

\bigskip
\noindent ABSTRACT

In cluster-specific studies, ordinary logistic regression and conditional logistic regression for binary outcomes provide maximum likelihood estimator (MLE) and conditional maximum likelihood estimator (CMLE), respectively. In this paper, we show that CMLE is approaching to MLE asymptotically when each individual data point is replicated infinitely many times. Our theoretical derivation is based on the observation that a term appearing in the conditional average log-likelihood function is the coefficient of a polynomial, and hence can be transformed to a complex integral by Cauchy's differentiation formula. The asymptotic analysis of the complex integral can then be performed using the classical method of steepest descent. 
Our result implies that CMLE can be biased if individual weights are multiplied with a constant, and that we should be cautious when assigning weights to cluster-specific studies. 

\vfill

\newpage

\section{Introduction}

\label{sec:mod_descr}

Consider a cluster-specific logistic model with $J$ clusters and $K_j$ individual data points in each cluster. The total number of individuals is
\begin{align}
	\label{eq:N}
	N:=\sum_{j=1}^{J}K_j.
\end{align}
We assume that there are $P$ individual level covariates, and denote the covariates and binary outcome of individual $k$ in cluster $j$ by $P$-dimensional vector $ \boldsymbol{X}_{j,k} = (X_{j,k,1},\cdots,X_{j,k,P})^T$ and binary scalar $ Y_{j,k} $, respectively.
The cluster-specific logistic model is set up as 
\begin{equation}\label{eq:model}
\text{logit}\{P(Y_{j,k}=1|\boldsymbol{X}_{j,k},b_j)\}=\boldsymbol{X}_{j,k}^T\boldsymbol{\beta}+b_j,\ \forall k = 1,\ldots,K_j,\ \forall j= 1,\ldots,J.
\end{equation}
Here, the parameters of interest are $ \boldsymbol{\beta} = (\beta_1,\cdots,\beta_P)^T\in\mathbb{R}^P $ and the nuisance parameter $ b_j $. Note that $b_j$ is a cluster-specific effect for the $ j $-th cluster. 
%
%
%
We make the cluster independence assumption in the logistic model \eqref{eq:model}.

\begin{assum}[Cluster independence]
	\label{as:indep}
	The clusters are independent of each other, i.e., 
	\begin{equation}\label{eq:cluster_indep}
	(\boldsymbol{Y}_j,\boldsymbol{X}_j,b_j)\independent (\boldsymbol{Y}_{j'},\boldsymbol{X}_{j'},b_{j'}) \quad \text{ for any }j\neq {j'}.
	\end{equation}
	Here $ \independent $ is the independence notation.
\end{assum}
The above assumption means that individuals from different clusters are independent of each other. Such assumption is consistent with our understanding of the data structure in the cluster-specific study designs. In the sequel we may refer to the above assumption as the ``cluster independence assumption''.

In the following sections, we review the ordinary and conditional logistic regression models, and then describe our problem of interest on the asymptotic property of conditional maximum likelihood estimators arising from survey sampling. We prove our main result in Section \ref{sec:results}. 
Detailed proof of a key proposition is described in Appendix \ref{sec:asym}.

\subsection{Ordinary and conditional logistic regression estimators}
\label{subsec:olr_mle}
The ordinary logistic regression model has the following average log-likelihood:
\begin{align}\label{eq:llhdMLEb}
\begin{aligned}
\overline{l}^o(\boldsymbol{\beta}, \boldsymbol{b}) 
= & \frac{1}{N}\sum_{j=1}^{J}\sum_{k=1}^{K_j}Y_{j,k}(\boldsymbol{X}_{j,k}^T\boldsymbol{\beta} + b_j) - \log\left[1+\exp(\boldsymbol{X}_{j,k}^T\boldsymbol{\beta} + b_j)\right].
\end{aligned}
\end{align}
Here recall from \eqref{eq:N} that we use $N:=\sum_{j=1}^{J}K_j$ to denote the total number of individuals. 
The estimating equations are
\begin{align}
	\label{eq:bjo}
	& \sum_{k=1}^{K_j}  \frac{\exp(\boldsymbol{X}_{j,k}^T\boldsymbol{\beta}+ b_j)}{1+ \exp(\boldsymbol{X}_{j,k}^T\boldsymbol{\beta} + b_j)} = \sum_{k=1}^{K_j}Y_{j,k},\ \forall j=1,\ldots,J,
	\\
	\label{eq:olr.ee}
	& \sum_{j=1}^{J}\sum_{k=1}^{K_j}\frac{\exp(\boldsymbol{X}_{j,k}^T\boldsymbol{\beta} +  b_{j})}{1+\exp(\boldsymbol{X}_{j,k}^T\boldsymbol{\beta} + b_{j})} \boldsymbol{X}_{j,k} = \sum_{j=1}^{J}\sum_{k=1}^{K_j}Y_{j,k}\boldsymbol{X}_{j,k}.
\end{align}
We assume that there exists a finite valued maximum likelihood estimator (MLE) $(\hat{\boldsymbol{\beta}}^o, \hat{b}^o)$. 
Note that for any finite-valued parameters $\vc{\beta}$ and $b$ we have $0<\exp(\boldsymbol{X}_{j,k}^T\boldsymbol{\beta}+ b_j)/[1+ \exp(\boldsymbol{X}_{j,k}^T\boldsymbol{\beta} + b_j)]<1$ in the left-hand-side of \eqref{eq:bjo}. Consequently, under our assumption that finite-valued MLE exists, we have $1\le \sum_{k=1}^{K_j}Y_{k}\le K_j$.

The average log-likelihood function $\overline{l}^o(\boldsymbol{\beta}, \boldsymbol{b})$ is concave. It is strictly concave under the following assumption:
\begin{assum}
	\label{as:fullrankX}
	All elements in $ \boldsymbol{X}_{j,k} $ are finite valued, and the $N\times (P+J)$ matrix consisting of row vectors $(\vc X_{j,k}^T, \vc e_j^T)^T$,  $k=1,\ldots,K_j$ and $j=1,\ldots,J$ has full column rank.
\end{assum}
\noindent In the sequel we may refer to the above assumption as the ``column rank assumption''. The strict concavity result under the above assumption is well known and can be verified by checking that the Hessian of function $\overline{l}^o(\boldsymbol{\beta}, \boldsymbol{b})$ is positive definite for any $(\boldsymbol{\beta}, \boldsymbol{b})$. Due to its strict concavity, there exists at most one maximizer of the function $\overline{l}^o(\boldsymbol{\beta}, \boldsymbol{b})$.

For our future discussion, we also adopt an equivalent description of the average log-likelihood function. Given any finite valued $\vc\beta$, observe that there exists unique $b_j$'s that satisfy the $J$ equations in \eqref{eq:bjo}, since the function $b_j\mapsto \sum_{k=1}^{K_j}{\exp(\boldsymbol{X}_{j,k}^T\boldsymbol{\beta}+ b_j)}/[{1+ \exp(\boldsymbol{X}_{j,k}^T\boldsymbol{\beta} + b_j)}]$ is strictly monotonically increasing. Therefore, for each cluster $j$ we can denote function $\tau_j(\vc\beta)$ to be the unique root to the $j$-th equation in the estimating equation \eqref{eq:bjo} with respect to the given $\vc\beta$. Consequently, solving MLE $(\hat{\boldsymbol{\beta}}^o, \hat{b}^o)$ by maximizing the average log-likelihood \eqref{eq:llhdMLEb} is equivalent to maximizing the following function:
\begin{align}
\label{eq:llhdMLE}
\begin{aligned}
l^o(\boldsymbol{\beta}) = & \frac{1}{N}\sum_{j=1}^{J}\sum_{k=1}^{K_j}Y_{j,k}(\boldsymbol{X}_{j,k}^T\boldsymbol{\beta} + \tau_j(\boldsymbol{\beta})) - \log\left[1+ \exp(\boldsymbol{X}_{j,k}^T\boldsymbol{\beta} + \tau_j(\boldsymbol{\beta}))\right].
\end{aligned}
\end{align}
The estimating equations become 
\begin{equation}\label{eq:olr.ee1}
\sum_{j=1}^{J}\sum_{k=1}^{K_j}Y_{j,k}\boldsymbol{X}_{j,k}-\sum_{j=1}^{J}\sum_{k=1}^{K_j}\frac{\exp\{\boldsymbol{X}_{j,k}^T\boldsymbol{\beta} +  \tau_j(\boldsymbol{\beta})\}}{1+\exp\{\boldsymbol{X}_{j,k}^T\boldsymbol{\beta} + \tau_j(\boldsymbol{\beta})\}} \boldsymbol{X}_{j,k} = \boldsymbol{0}.
\end{equation}
Since function $\overline{l}^o(\boldsymbol{\beta}, \boldsymbol{b})$ has at most a unique maximizer under Assumption \ref{as:fullrankX}, we can conclude that there also exists at most one unique maximizer of function $l^o(\boldsymbol{\beta})$ in \eqref{eq:llhdMLE}.

The CMLE of conditional logistic regression is obtained by maximizing the overall conditional likelihood composed of joint probabilities conditional on sufficient statistics. For model \eqref{eq:model}, the sufficient statistic for $ b_j $ in model \eqref{eq:model} is the outcome sum $ \sum_{k=1}^{K_j}Y_{j,k} $ in the $ j $-th cluster. Letting $ \boldsymbol{y_j}=(y_{j,1},y_{j,2},\cdots,y_{j,K_j})^T $ be the observed cluster level outcome in the $ j $-th cluster, the conditional joint probability  is
\begin{align}
\label{eq:clr.joint_prob}
\begin{aligned}
P\left(\boldsymbol{Y}_j=\boldsymbol{y}_j\Big|\sum_{k=1}^{K_j}Y_{j,k}=\sum_{k=1}^{K_j}y_{j,k}\right) 
= &
\dfrac{\prod_{k=1}^{K_j}P(Y_{j,k}=y_{j,k}|\boldsymbol{X}_{j,k},b_j)}{\sum_{\boldsymbol{r}\in \mathbb{V}_j} \prod_{k=1}^{K_j}P(Y_{j,k}=r_{j,k}|\boldsymbol{X}_{j,k},b_j)} \\
= & 
\dfrac{\prod_{k=1}^{K_j}\exp(y_{j,k}\boldsymbol{X}_{j,k}^T\boldsymbol{\beta})}{\sum_{\boldsymbol{r} \in \mathbb{V}_j}\prod_{k=1}^{K_j}\exp(r_{j,k}\boldsymbol{X}_{j,k}^T\boldsymbol{\beta})}.
\end{aligned}
\end{align}
Here 
$ \mathbb{V}_j = \Set{\vc r:=(r_{1},\cdots,r_{K_j})^T\in \{0,1\}^{K_j}|\sum_{k=1}^{K_j}r_{k}=\sum_{k=1}^{K_j}y_{j,k}} $ 
is the set of 
all possible permutations of outcomes in the $j$-th cluster. Note from the above equation that the cluster level conditional joint probability does not depend on $ b_j $. Also, when $\sum_{k=1}^{K_j}Y_{j,k}$ is either $0$ or $K_j$, the conditional joint probability is trivial: $ P(\boldsymbol{Y}_j=\boldsymbol{0}|\sum_{k=1}^{K_j}Y_{j,k}=0) = P(\boldsymbol{Y}_j=\boldsymbol{1}|\sum_{k=1}^{K_j}Y_{j,k}=K_j)=1 $. Under the assumption that $1\le \sum_{k=1}^{K_j}Y_{j,k}\le K_j-1$, the conditional joint probability in \eqref{eq:clr.joint_prob} can be considered as the probability of a multivariate Fisher's noncentral hypergeometric distribution, where the exponential form $ \exp(\boldsymbol{X}_{j,k}^T\boldsymbol{\beta}) $ is treated as a weight  \citep{fisher1935logic,cornfield1956statistical,agresti1992survey}. 

Based on the conditional joint probability in \eqref{eq:clr.joint_prob}, we obtain the corresponding overall 
average conditional log-likelihood 
\begin{align}
l^c(\boldsymbol{\beta}) = \frac{1}{N}\left\{\sum_{j=1}^{J}\sum_{k=1}^{K_j}Y_{j,k}\boldsymbol{X}_{j,k}^T\boldsymbol{\beta} - \sum_{j=1}^{J}\log\left[\sum_{ \boldsymbol{r}\in \mathbb{V}_j}\exp\left(\sum_{k=1}^{K_j}r_{k}\boldsymbol{X}_{j,k}^T\boldsymbol{\beta}\right)\right]\right\}.
\end{align}
The above function is concave since it is the sum of a linear function and the negative of a log-sum-exp function (see, e.g., \cite{boyd2004convex} for the convexity of log-sum-exp functions). 
A conditional maximum likelihood estimator (CMLE), denoted as $ \hat{\boldsymbol{\beta}}_c $, can be obtained by solving the estimating equations
\begin{equation}
\label{eq:clr.ee1}
\sum_{j=1}^{J}\sum_{k=1}^{K_j}Y_{j,k}\boldsymbol{X}_{j,k}-\sum_{j=1}^{J}\sum_{\substack{k=1\\\boldsymbol{r}^{*} \in \mathbb{V}_j}}^{K_j}\frac{\exp (\sum_{k=1}^{K_j}r^{*}_{k}\boldsymbol{X}_{j,k}^T\boldsymbol{\beta})}{\sum_{\boldsymbol{r} \in \mathbb{V}_j}\exp(\sum_{k=1}^{K_j}r_{k}\boldsymbol{X}_{j,k}^T\boldsymbol{\beta})}r^{*}_{k}\boldsymbol{X}_{j,k}=\boldsymbol{0}.
\end{equation}

We finish this subsection with a discussion on the relationship between MLE and CMLE. By comparing the estimating equations (\ref{eq:olr.ee1}) and (\ref{eq:clr.ee1}), we have
\begin{equation}\label{eq:clr.olr.relat}
\sum_{j=1}^{J}\sum_{k=1}^{K_j}\frac{\exp\{\boldsymbol{X}_{j,k}^T\boldsymbol{\hat{\beta}^o} + b_{j}(\boldsymbol{\hat{\beta}^o})\}}{1+\exp\{\boldsymbol{X}_{j,k}^T\boldsymbol{\hat{\beta}^o} + b_{j}(\boldsymbol{\hat{\beta}^o})\}} \boldsymbol{X}_{j,k}
= 
\sum_{j=1}^{J}\sum_{\substack{k=1\\\boldsymbol{r}^{*} \in \mathbb{V}_j}}^{K_j}\frac{\exp (\sum_{k=1}^{K_j}r^{*}_{k}\boldsymbol{X}_{j,k}^T\boldsymbol{\beta})}{\sum_{\boldsymbol{r} \in \mathbb{V}_j}\exp(\sum_{k=1}^{K_j}r_{k}\boldsymbol{X}_{j,k}^T\boldsymbol{\beta})}r^{*}_{k}\boldsymbol{X}_{j,k}.
\end{equation}
There are a few previous results in the literature that are special cases of the above relation.
Specifically, for a matched-pair design with no covariates other than the treatment/control information (i.e., $ P=1 $ and $K_j\equiv 2$), it has been proved by \citet[pp. 244-245]{andersen1980discrete} and \citet[pp. 493]{agresti2012categorical} that \eqref{eq:clr.olr.relat} becomes $ {\hat{\beta}^o}=2{\hat{\beta}^c} $. For a matched-pair design with more covariates (i.e., $ P>1 $ and $K_j\equiv 2$), in \cite{he2013equivalence} and \cite{he2012causal}, it is proved that \eqref{eq:clr.olr.relat} becomes $ \boldsymbol{\hat{\beta}^o}=2\boldsymbol{\hat{\beta}^c} $. For a $1:K$ matched treatment-control design ($ K>1 $) with no covariates other than the treatment/control information (i.e., $ P=1 $ and $K_j\equiv 1+K$), it is shown in \cite{he2018consistency} that \eqref{eq:clr.olr.relat} reduces to
\begin{equation}\label{eq:clr.olr.relat1}
\sum_{t=1}^{K}\frac{n_{t}}{1+\frac{t}{K-t+1}e^{\hat{\beta}^c}}	
=\sum_{t=1}^{K}\frac{n_{t}}{1+\frac{t-1}{2(K-t+1)}e^{\hat{\beta}^o}-\frac{K-t}{2(K-t+1)}+\frac{1						}{2(K-t+1)}\sqrt{\Delta_{K,t}(\hat{\beta}^o)}}, 
\end{equation}
where $ n_t$ is the number of clusters satisfying $ \sum_{k=1}^{K_j}Y_{j,k}=t $,
and $ \Delta_{K,t}(\hat{\beta}^o):=[(t-1)e^{\hat{\beta}^o}-(K-t)]^2+4t(K_j-t)e^{\hat{\beta}^o} $. 


\subsection{Problem of interest: MLE and CMLE in survey sampling}
\label{subsec:surveysamp}
Let us consider MLE and CMLE in survey sampling. If individuals in the data are sampled from a target population, we have to implement sampling weights in the estimating equations, since the estimating procedures are targeted at the larger population rather than the sample. 
One may understand the sampling weight structure as repeated observations at each data point \citep{he2013equivalence,he2018consistency}. 
Let us consider a specific sampling weight structure that is equivalent to $R$ replications of samples. 
After replication, in the $j$-th cluster we have $RK_j$ data points $(\vc X_{j,k}^{(t)}, Y_{j,k}^{(t)})$, $t=1,\ldots,R$ and $k=1,\ldots,K_j$ in which 
\begin{align}
	\label{eq:rep}
	{\vc X}_{j,k}^{(1)}=\cdots ={\vc X}_{j,k}^{(R)}={\vc X}_{j,k}\text{ and }Y_{j,k}^{(1)}=\cdots =Y_{j,k}^{(R)}=Y_{j,k},\ 
	\forall k,j.
\end{align}
The total number of data points is now $NR$, where recall from \eqref{eq:N} that $N=\sum_{j=1}^{J}K_j$.

Applying ordinary logistic regression to the replicated data points, by \eqref{eq:llhdMLEb} the average log-likelihood is
\begin{align}
	\overline{l}^o(\boldsymbol{\beta}, \boldsymbol{b}) 
	= & \frac{1}{NR}\sum_{j=1}^{J}\sum_{k=1}^{K_j}\sum_{t=1}^{R}Y_{j,k}^{(t)}\left(\left(\boldsymbol{X}_{j,k}^{(t)}\right)^T\boldsymbol{\beta} + b_j\right) - \log\left[1+\exp\left(\left(\boldsymbol{X}_{j,k}^{(t)}\right)^T\boldsymbol{\beta} + b_j\right)\right]
	\\
	= & \frac{1}{N}\sum_{j=1}^{J}\sum_{k=1}^{K_j}Y_{j,k}(\boldsymbol{X}_{j,k}^T\boldsymbol{\beta} + b_j) - \log\left[1+\exp(\boldsymbol{X}_{j,k}^T\boldsymbol{\beta} + b_j)\right].
\end{align}
Here the last equality is from \eqref{eq:rep}. Therefore, the above average log-likelihood remains the same as \eqref{eq:llhdMLE}.

For conditional logistic regression, the average conditional log-likelihood over the replicated data becomes
\begin{align}
l^c(\boldsymbol{\beta}) = &  \frac{1}{NR}\left\{\sum_{j=1}^{J}\sum_{k=1}^{K_j}\sum_{t=1}^{R}Y_{j,k}^{(t)}\left(\boldsymbol{X}_{j,k}^{(t)}\right)^T\boldsymbol{\beta} - \sum_{j=1}^{J}\log\left[\sum_{ \boldsymbol{r}\in \mathbb{V}_j}\exp\left(\sum_{k=1}^{K_j}\sum_{t=1}^{R}r_{k}^{(t)}\left(\boldsymbol{X}_{j,k}^{(t)}\right)^T\boldsymbol{\beta}\right)\right]\right\}
\\
= & \frac{1}{NR}\left\{R\sum_{j=1}^{J}\sum_{k=1}^{K_j}Y_{j,k}\boldsymbol{X}_{j,k}^T\boldsymbol{\beta} - \sum_{j=1}^{J}\log\left[\sum_{ \boldsymbol{r}\in \mathbb{V}_j}\exp\left(\sum_{k=1}^{K_j}\left(\sum_{t=1}^{R}r_{k}^{(t)}\right)\boldsymbol{X}_{j,k}^T\boldsymbol{\beta}\right)\right]\right\}.
\end{align}
Here we use \eqref{eq:rep} to obtain the last equality. The set $\mathbb{V}_j$ now is
\begin{align}
& \mathbb{V}_j 
\\
= & \Set{\vc r:=\left(r_{1}^{(1)},r_{1}^{(2)},\cdots, r_{1}^{(R)}, \ldots, r_{K_j}^{(1)},r_{K_j}^{(2)},\cdots,r_{K_j}^{(R)}\right)\in \{0,1\}^{RK_j}|\sum_{k=1}^{K_j}\sum_{t=1}^{R}r_{k}^{(t)}=R\sum_{k=1}^{K_j}y_{j,k}}.
\end{align}
Let us fix a specific cluster $j$. Introducing $r_{k}:=\sum_{t=1}^{R}r_{k}^{(t)}$ we have $r_{k}\in \{0,1,\ldots,R\}$ and $\sum_{k=1}^{K_j}r_{k} = R\sum_{k=1}^{K_j}Y_{k,j}$. Moreover, observe that for any $r_{k}$, there are in total $R\choose r_{k}$ combinations of binary values $ r_{K_j}^{(1)},r_{K_j}^{(2)},\cdots,r_{K_j}^{(R)}$ such that $\sum_{t=1}^{R}r_{k}^{(t)}=r_{k}$. Counting all the combinations when performing the sum over $\vc r\in \mathbb{V}_j$, we have
\begin{align}
& \sum_{ \boldsymbol{r}\in \mathbb{V}_j}\exp\left(\sum_{k=1}^{K_j}\left(\sum_{t=1}^{R}r_{k}^{(t)}\right)\boldsymbol{X}_{j,k}^T\boldsymbol{\beta}\right) 
\\
= & \sum_{\substack{r_{1},\ldots,r_{K_j}\in\{0,1,\ldots,R\}\\\sum_{s=1}^{K_j}r_{s} = R\sum_{s=1}^{K_j}y_{j,s}}}{R\choose r_{1}}\ldots{R\choose r_{K_j}}\exp\left(\sum_{k=1}^{K_j}\left(\sum_{t=1}^{R}r_{k}^{(t)}\right)\boldsymbol{X}_{j,k}^T\boldsymbol{\beta}\right).
\end{align}
Applying the above result, we obtain the following description of conditional average log-likelihood over $R$ replications of data points:
\begin{align}\label{eq:llhdCMLE}
l^c_R(\boldsymbol{\beta}) = \frac{1}{RN}\left[R\sum_{j=1}^{J}\sum_{k=1}^{K_j}Y_{j,k}\boldsymbol{X}_{j,k}^T\boldsymbol{\beta} - \sum_{j=1}^{J}\log g_{j,R,\sum_{k=1}^{K_j}Y_{j,k}}(\boldsymbol{\beta})\right],
\end{align}
where
\begin{align}\label{eq:g}
g_{j,R,T}(\boldsymbol{\beta}):=\sum_{\substack{r_1,\ldots,r_{K_j}\in\{0,1,\ldots,R\}\\ \sum_{s=1}^{K}r_s=RT}} {R\choose r_1}\ldots{R\choose r_{K_j}}\exp\left(\sum_{k=1}^{K_j}r_k\boldsymbol{X}_{j,k}^T\boldsymbol{\beta}\right).
\end{align}
Clearly, when we change the number of replications $R$, the conditional average log-likelihood function $l^c_R(\vc\beta)$ and the associated estimation equation will also change accordingly. The main goal of this paper is answer the following research question:

{\bf Research questions: } What is the asymptotic property of the average log-likelihood $l^c_R(\vc\beta)$ when $R\to\infty$? What is the relation between the CMLE $\hc{\beta}^c_R$ and MLE $\hc{\beta}^o$ when $R\to\infty$?

\subsection{A motivating example}\label{subsec:simu}
In this section, we conduct a simulation study to demonstrate how the aforementioned replication $R$ may affect the performance of the MLE and CMLE.
Our example is based on the following logistic model:
\begin{equation}\label{eq:simu_model}
\text{logit}\{P(Y_{j,k}=1|\boldsymbol{X}_{j,k},b_j)\}=b_j+X_{j,k,1}\beta_1+X_{j,k,2}\beta_2.
\end{equation}
In this model, there are two covariates for the $k$-th individual in the $j$-th cluster. The number of individuals in any $j$-th cluster is set to $K_j\equiv K$. For any $j$ and $k$, the first covariate $ X_{j,k,1} $ describes the treatment/control assignment. We set $ X_{j,1,1}=1$ and $X_{j,2,1}=\cdots=X_{j,K_j,1}=0 $, which indicates a $1:(K-1)$ matched treatment-control design: the first individual in each cluster always receives the treatment and the rest always receive control. The second covariate $ X_{j,k,2} $ is randomly generated through an i.i.d. standard normal distribution. For any $j$, we set the cluster-specific effect in the $j$-th cluster to $ b_j=\delta_j -5\bar{X}_{j,\cdot,1}+3\bar{X}_{j,\cdot,2}$, where $ \delta_j \sim N(0,1) $ and $ \bar{X}_{j,\cdot,1} $ and $ \bar{X}_{j,\cdot,2} $ are the means of the first and second covariates in the $ j $-th cluster respectively. The value of the true parameter is set to $ (\beta_1,\beta_2)=(0.5, 0.8) $. Each simulation repeats 10,000 times. 

In the simulation, we set $ K_j\equiv K = 3 $ and $J=100$ to indicate a cluster-specific $1 : 2$ matched treatment-control design with $100$ clusters in each replication of simulation. 
As discussed in Section \ref{subsec:surveysamp}, the MLE is invariant with respect to the number of replications $ R $;
the estimated value of the MLE is $ (\hat{\beta}^o_1, \hat{\beta}^o_2)=(0.801, 1.283) $ with estimated variance = (0.383, 0.265). The CMLEs with respect to different numbers of replications $R$ are listed in Table \ref{tab:simu_R}.
\begin{table}[h]
	\centering
	\begin{tabular}{rcc} \hline
		& CMLE & var. \\
		$ R $   & ($\hat{\beta}^c_1, \hat{\beta}^c_2$) & of CMLE  \\ \hline
		1   & ({0.504, 0.812}) & ({0.238, 0.158})  \\
		2   & (0.630, 1.010) & (0.297, 0.194)    \\
		3   & (0.683, 1.094) & (0.323, 0.212)    \\
		4   & (0.711, 1.140) & (0.337, 0.224)    \\
		5   & (0.729, 1.168) & (0.346, 0.231)    \\
		10  & (0.765, 1.225) & (0.364, 0.248)    \\
		15  & (0.777, 1.245) & (0.371, 0.254)    \\
		20  & (0.783, 1.254) & (0.374, 0.257)    \\
		50  & (0.794, 1.271) & (0.379, 0.262)    \\
		{80} & ({0.796, 1.276}) & ({0.381, 0.263})    \\ 
		\hline
	\end{tabular}\\
	\caption{\label{tab:simu_R}Simulation study to investigate asymptotic property of the CMLE with respect to the number of replications $ R $. The true parameter is $ (\beta_1,\beta_2)=(0.5, 0.8) $. The MLE for all versions of $R$ in the simulation is $ (\hat{\beta}^o_1, \hat{\beta}^o_2)=(0.801, 1.283) $ with estimated variance = (0.383, 0.265).}
\end{table}	

A few remarks are in place. 
First, the simulation result suggests that the MLE is biased. 
Second, comparing all CMLEs with respect to the number of replications $R$, the simulation result for the CMLE with $ R=1 $ has the smallest bias, and the bias increases as $ R $ increases. 
Finally, the difference between the MLE and CMLE is becoming smaller as $ R $ increases. In fact, we observe from this simulation an interesting behavior of the CMLE with respect to different values of $R$: when $R=1$, the value of the CMLE is far from that of the MLE, and the bias of the CMLE is small; as $R$ increases, the value of the CMLE is converging to that of the MLE and is showing significantly more bias.

\section{Main result}\label{sec:results}

Our goal in this paper is to show that the CMLE converges to the MLE under the cluster independence and column rank assumptions:
\begin{theorem}\label{thm:r}
	For the cluster-specific logistic model described in (\ref{eq:model}), let $ \hat{\boldsymbol{\beta}}^c(R) $ and $\hc{\beta}^o$ denote the CMLE with number of replications $ R $ and the MLE respectively. Under Assumptions \ref{as:indep} and \ref{as:fullrankX}, 
	for any fixed $J$ and $K_j$'s, if the CMLEs $ \hat{\boldsymbol{\beta}}^c(R) $ and MLE $ \hat{\boldsymbol{\beta}}^o $ exist, we have that $ \hat{\boldsymbol{\beta}}^c(R) $ approaches $ \hat{\boldsymbol{\beta}}^o $ as $ R $ approaches infinity. 
\end{theorem}
A few remarks are in place for the above theorem.
First, based on the above result, we are now able to develop a new perspective in terms of the relationship between the MLE and CMLE. In particular, the MLE can be understood as an extreme case of the CMLE in which the sample is replicated infinitely many times. Equivalently, we can also conclude that the conditional logistic regression with a large number of data point replications is closely related to the ordinary logistic regression. 
Second, the above theorem indicates that we should be extremely careful with the number of replications when performing conditional logistic regression for survey sampling. As observed in Section \ref{subsec:simu}, the MLE is a biased estimator that is invariant to the  $ R $. Therefore, Theorem \ref{thm:r} indicates that the CMLE can be potentially biased when the chosen number of replications is large. 
Finally, the theorem can be relaxed to the case when only a subsequence of CMLEs $\{\hat{\boldsymbol{\beta}}^c(R)\}_{R=1}^{\infty}$ exist.

Before describing the proof of Theorem \ref{thm:r}, we discuss our strategy of the proof and the technical challenge within. Briefly speaking, the key component of our analysis is the discovery of pointwise convergence of functions $l^c_R(\vc\beta)$ to $l^o(\vc\beta)$ as $R\to\infty$. Theorem \ref{thm:r} then follows immediately by noting that concavity and uniqueness of MLE improves such pointwise convergence to uniform convergence \citep[see, e.g., Theorem II.1 in ][] {andersen1982cox,rockafellar1970convex}. 
Specifically, we can apply Corollary II.2 in \cite{andersen1982cox}, which states the following:
\begin{pro}
	\label{pro:consistency}
	Consider a sequence of finite-valued concave functions $\{M_R(\vc\beta)\}_{R=1}^\infty$ that converge pointwisely to a function $M(\vc\beta)$.
	Suppose that a unique maximizer $\vc\beta^*$ exists for $M(\vc\beta)$, and that $\{\hc\beta_R\}$ is a sequence of maximizers of $M_R(\vc\beta)$. We have $\hc\beta_R\to \vc\beta^*$ when $i\to\infty$. Also, if we change the pointwise convergence assumption to pointwise convergence in probability, i.e., if $\{M_R(\vc\beta)\}_{i=1}^\infty$ are random functions that converge pointwisely in probability to $M(\vc\beta)$, then we have $\hc\beta_R\pto \vc\beta^*$. 
\end{pro}
\noindent The main challenge in our analysis is on the pointwise convergence of $l^c_R(\vc\beta)$ to $l^o(\vc\beta)$ as $R\to\infty$. Specifically, we need to analyze the behavior of the function $g_{j,R,T}(\vc\beta)$ in \eqref{eq:g} as $R\to\infty$, which involves asymptotic analysis of the sum of the multiplication of several combinatorial factors. Fortunately, through an application of the method of steepest descent for integral approximation, we can obtain the following proposition:
\begin{pro}
	\label{pro:fAsymp}
	For any fixed vector $\vc\beta$, positive integer $K$, and points $X_1,\ldots,X_K\in\R^P$, consider function $f:\R^P\to\R$ described by
	\begin{align}
	\label{eq:fR}
	f_{R}(\vc \beta) = \sum_{\substack{r_1,\ldots,r_{K}\in\{0,1,\ldots,R\}\\ \sum_{s=1}^{K}r_s=RT}} {R\choose r_1}\ldots{R\choose r_{K}}\exp\left(\sum_{k=1}^{K}r_k\vc X_{k}^T\vc \beta\right)
	\end{align}
	where $0\le T\le K$ is a fixed positive integer. We have
	\begin{align}
	\label{eq:fAsymp}
	\begin{aligned}
	& e^{-Ru(0)}f_{R}(\vc\beta) = O_p(1)R^{-1/2} + o_p(1)
	\end{aligned}	
	\end{align}	
	and consequently
	\begin{align}
	\label{eq:fAsymplim}
	\lim_{R\to\infty}\frac{1}{R}\log f_{R}(\vc\beta) = u(0).
	\end{align}
	Here $ O_p(1) $ denotes any term that is boudned in probability with respect to $R$, $ o_p(1) $ denotes any terms that converges to $0$ in probability when $ R\to\infty $, and $u(0)$ is a real valued constant evaluated at $\theta=0$ of the following complex-valued function $u(\theta)$:
	\begin{align}
	\label{eq:u}
	u(\theta) = -\tau(\vc\beta)T- i(K-T)\theta + \sum_{k=1}^{K}\log\left(e^{i\theta}+\exp(\vc X_k^T\vc\beta+\tau(\vc\beta)) \right),
	\end{align}	
	in which $i=\sqrt{-1}$ and $\tau(\vc\beta)$ is implicitly defined as the root $\tau$ to the following equation:
	\begin{align}
	\label{eq:b}
	T - \sum_{k=1}^K\frac{\exp(\vc X_k^T\vc\beta+\tau)}{1 + \exp(\vc X_k^T\vc\beta+\tau)} = 0.
	\end{align}	
\end{pro}
\noindent The application of the above proposition to the function $g$ in \eqref{eq:llhdCMLE} yields the pointwise convergence of $l^c_R(\vc\beta)$ to $l^o(\vc\beta)$. To the best of our knowledge, the above result is nontrivial and has not been discovered in the literature. Our proof is inspired by \citep{halasz1976elementary} in which the special case with $T=1$ is proved. The detailed proof will be postponed to Section \ref{sec:asym} in the Appendix.

We are now ready to prove Theorem \ref{thm:r}.
\begin{proof}[Proof of Theorem \ref{thm:r}]
	For any fixed $\vc\beta$ and $j=1,\ldots,J$, applying Proposition \ref{pro:fAsymp} to the description of $l^c_R(\vc\beta)$ in \eqref{eq:llhdCMLE} (with $K=K_j$, $T = \sum_{k=1}^{K_j}Y_{j,k}$, $\vc X_k =\vc X_{j,k}$, $f_{R}(\vc\beta) = g_{j,R,\sum_{k=1}^{K_j}Y_{j,k}}(\vc\beta)$, and $\tau(\vc\beta) = \tau_{j}(\vc\beta)$), we have 
	\begin{align}
	\lim_{R\to\infty}\frac{1}{R}\log g_{j,R,\sum_{k=1}^{K}Y_{j,k}}(\vc\beta) = u(0) := -\tau_j(\vc\beta)\sum_{k=1}^{K_j}Y_{j,k} + \sum_{k=1}^{K_j}\log(1 + \exp(\vc X_{j,k}^T\vc\beta+\tau_{j}(\vc\beta))).
	\end{align}
	Noting the above and the descriptions of $l^o(\vc\beta)$ and $l_{R}^c(\vc\beta)$ in \eqref{eq:llhdMLE} and \eqref{eq:llhdCMLE} respectively, we have
	\begin{align}
	& \lim_{R\to\infty} \left[l^o(\vc\beta) - l_{R}^c(\vc\beta)\right]
	\\
	= & \lim_{R\to\infty}\frac{1}{N}\sum_{j=1}^{J}\left[\sum_{k=1}^{K_j}Y_{j,k}\tau_j(\vc\beta) - \log(1+\exp(\vc X_{j,k}^T\vc \beta + b_{j,K}^o)) + \frac{1}{R}\log g_{j,R,K,\sum_{k=1}^{K}Y_{j,k}}(\vc\beta) \right] = 0.
	\end{align}
	The above result shows the pointwise convergence $l_{R}^c(\vc\beta)\to l^o(\vc\beta)$ as $R\to\infty$. Applying Proposition \ref{pro:consistency} (with $M_R(\vc\beta) = l_{R}^c(\vc\beta)$, $\hc\beta_R= \hc\beta^c_R$, $M(\vc\beta) = l^o(\vc\beta)$, and $\vc\beta^*=\hc\beta^o$ ), we conclude that $\hc\beta^c_R\to \hc\beta^o$ as $R\to\infty$.
\end{proof}

\section{Discussion}\label{sec:discussion}
In this paper, we discuss the performance of both ordinary logistic regression and conditional logistic regression methods when each individual data point is replicated many times. Specifically, the conditional logistic regression is asymptotically equivalent to the ordinary logistic regression for the infinitely replicated sample, i.e. $\lim\limits_{R\rightarrow\infty}\boldsymbol{\hat{\beta}}_R^c=\boldsymbol{\hat{\beta}}^o $. Consequently, when the individual data points are replicated, conditional logistic regression results in a biased estimation. Noting that the replication of data points is in a sense related to the design of sampling weights \citep{he2013equivalence,he2018consistency}, 
our results implies that we should be cautious when working on weight structure to the samples.

\section{Acknowledgment}

Part of Yuyuan Ouyang's research is supported by US Dept. of the Air Force grant FA9453-19-1-0078 and ONR grant N00014-19-1-2295.

\bibliographystyle{chicago}
\bibliography{reference-OLR-CLR-1overM}

\begin{center}
{\large\bf APPENDIX}
\end{center}

\begin{appendix}
	\Appendix    
	\renewcommand{\theequation}{A.\arabic{equation}}
	\renewcommand{\thesubsection}{A.\arabic{subsection}}
	\setcounter{equation}{0}
\end{appendix}

\section{Proof of Proposition \ref{pro:fAsymp}}
\label{sec:asym}
Our goal in this section is to prove Proposition \ref{pro:fAsymp}. The key concept is to notice that the function $f_{R}(\vc\beta)$ is indeed the coefficient of a polynomial and can be transformed to a complex integral by apply Cauchy's differentiation formula. The classical method of steepest descent can then be applied to analyze the integral, yielding the asymptotic estimate of $f_R(\vc\beta)$.

	\begin{proof}[Proof of Proposition \ref{pro:fAsymp}]
		Letting $\xi_k = \exp(\vc X_k^T\vc \beta)$ and computing the coefficients of the polynomial $\prod_{k=1}^{K}(z+\xi_k)^R$ with respect to its variable $z$, we have
		\begin{align}
		\prod_{k=1}^{K}(z+\xi_k)^R = \sum_{l=0}^{KR}\left[\sum_{\substack{r_1,\ldots,r_{K}\in\{0,1,\ldots,R\}\\ \sum_{s=1}^{K}r_s=l}} {R\choose r_1}\ldots{R\choose r_{K}}\xi_1^{r_1}\cdots\xi_K^{r_K}\right]z^{KR-l}.
		\end{align}
		Comparing the above expansion with the definition of $f_{R,K}(\vc\beta)$ in \eqref{eq:fR}, we observe that $f_{R,K}(\vc\beta)$ is indeed the coefficient of $z^{R(K-T)}$ in the above expansion. For any $\rho>0$, by applying Cauchy's differeniation formula to the above polynomial over the disk $\{z:|z|\le\rho\}$ we have
		\begin{align}
		\label{eq:fIntegralh}
		\begin{aligned}
		f_{R}(\vc\beta) = & \frac{1}{2\pi i}\oint_{|z|=\rho}\dfrac{\prod_{k=1}^{K}(z+\xi_k)^R}{z^{R(K-T)+1}}dz
		\\
		= & \frac{1}{2\pi i}\int_{-\pi}^{\pi}(\rho e^{i\theta})^{-R(K-T)-1}\exp\left[R\sum_{k=1}^{K}\log \left(\rho e^{i\theta}+\xi_k\right)\right]\rho ie^{i\theta}~d\theta
		\\
		= & \frac{1}{2\pi}\int_{-\pi}^{\pi}e^{h(\rho ,\theta)}~d\theta,\ \forall \rho>0,
		\end{aligned}
		\end{align}
		where in the last equality by simplifying terms and noting that $\xi_k = \exp(\vc X_k^T\vc \beta)$ we have
		\begin{align}
		\label{eq:h}
		h(\rho ,\theta) = -R(K-T)\log \rho  - iR(K-T)\theta + R\sum_{k=1}^{K}\log(\rho e^{i\theta} + \exp(\vc X_k^T\vc\beta)).
		\end{align}
		
		In order to analyze the asymptotic behavior of $f_{R}(\vc\beta)$ in the integral form \eqref{eq:fIntegralh}, we apply the classical method of steepest descent (see, e.g., \cite{wong2001asymptotic}) to the integral. In particular, noting that \eqref{eq:fIntegralh} holds for any $\rho>0$, the method of steepest descent suggests a specific value of $\rho=e^{-\tau(\vc\beta)}$ where $\tau(\vc\beta)$ is defined by equation \eqref{eq:b}, so that $\partial_{\theta}h(\rho ,0)=0$. Indeed, we can verify that when $\rho=e^{-\tau(\vc\beta)}$, using the definition of $\tau(\vc\beta)$ in \eqref{eq:b} we have
		\begin{align}
		\label{eq:partialh}
		\begin{aligned}
		& \partial_{\theta}h\left(e^{-\tau(\vc\beta)} ,0\right) = -iR(K-T) + iR\sum_{k=1}^{K}\frac{e^{-\tau(\vc\beta)}}{e^{-\tau(\vc\beta)}  + \exp(\vc X_k^T\vc\beta)} 
		\\
		=& iRT + iR\sum_{k=1}^{K}\left(1-\frac{1}{1  + \exp(\vc X_k^T\vc\beta + \tau(\vc\beta))}\right) = 0.
		\end{aligned}
		\end{align}		
		Setting $\rho=e^{-\tau(\vc\beta)}$ in \eqref{eq:fIntegralh} and noting the description of $h(\rho,\theta)$ in \eqref{eq:h} we obtain
		\begin{align}
		\label{eq:fIntegral}
		f_{R}(\beta) = \frac{1}{2\pi}\int_{-\pi}^{\pi}e^{Ru(\theta)}~d\theta,
		\end{align}
		where $u(\theta) =  h(e^{-\tau(\vc\beta)},\theta)/R $. It is easy to verify that $u(\theta)$ has form \eqref{eq:u}.
		Note that with the specific choice $\rho=e^{-\tau(\vc\beta)}$ we have $u'(0)=\partial_{\theta}h\left(e^{-\tau(\vc\beta)} ,0\right)/R=0$.
		
		To study the behavior of the integral \eqref{eq:fIntegral} when $R\to\infty$, we separate it to two parts with
		\begin{align}
		I_1 := \frac{1}{2\pi}\int_{-\delta}^{\delta}e^{Ru(\theta)}~d\theta\text{ and }I_2 := \frac{1}{2\pi}\int_{[-\pi,\pi]\backslash[-\delta,\delta]}e^{Ru(\theta)}~d\theta,
		\end{align}
		where $\delta\in(0,\pi/2)$ is a constant that will be determined later. We will analyze integrals $I_1$ and $I_2$ separately. 
		
		The analysis of $I_1$ requires the following Taylor expansion of $u(\theta)$  at $\theta=0$:
		\begin{align}
		\label{eq:uTaylor}
		\begin{aligned}
		u(\theta) = & u(0) + u'(0)\theta + \frac{1}{2}u''(0)\theta^2 + \frac{1}{6}u'''(\theta^*)\theta^3 =  u(0) - O_p(1)\theta^2 + O_p(1)\delta^3.
		\end{aligned}
		\end{align}		
		Here $\theta^*\in[-\delta,\delta]$ is used in the remainder term of the Taylor expansion. The last equality is since $u'(0)=0$ and the observation that $u(\theta)$ does not depend on $R$.
%
		Applying the reformulation of $u(\theta)$ in \eqref{eq:uTaylor} to the definition of $I_1$ we have
		\begin{align}
		\label{eq:I1result}
		\begin{aligned}
		I_1 = & \frac{1}{2\pi}e^{Ru(0)+O_p(1)R\delta^3}\int_{-\delta}^{\delta}\exp\left(-O_p(1)R\theta^2\right)d\theta 
		\\
		= & \frac{1}{2\pi}e^{Ru(0)+O_p(1)R \delta^3}\left[\int_{-\infty}^{\infty}\exp\left(-O_p(1)R\theta^2\right)d\theta - 2\int_{\delta}^{\infty}\exp\left(-O_p(1)R\theta^2\right)d\theta\right]
		\\
		= & \frac{1}{2\pi}e^{Ru(0)+O_p(1)R\delta^3}
		\left[
		O_p(1)R^{-1/2}
		- \int_{\delta}^{\infty}\exp\left(-O_p(1)R\theta^2\right)d\theta
		\right]
		\\
		= & e^{Ru(0)+O_p(1)R\delta^3}
		\left[
		O_p(1)R^{-1/2}
		-O_p(1)R^{-1}\delta^{-1}\exp\left(-O_p(1)R\delta^2\right)
		\right].
		\end{aligned}		
		\end{align}	
		Here the last equality is due to the boundedness result
		\begin{align}
		\int_{\delta}^{\infty}\exp\left(-O(1)R\theta^2\right)d\theta \le & \int_{\delta}^{\infty}\exp\left(-O(1)R\delta\theta\right)d\theta =  O(1)R^{-1}\delta^{-1}\exp\left(-O(1)R\delta^2\right).
		\end{align}	
		
		Let us continue to the integral $I_2$. Noting the description of $u(\theta)$ and $u(0)$ (see \eqref{eq:u}) we have that for any $\theta\in [-\pi,\pi]\backslash[-\delta,\delta]$ where $0<\delta<\pi$, 
		\begin{align}
		|e^{Ru(\theta)}| = & e^{Ru(0)}\left|e^{-iR(K-T)\theta}\right|\prod_{k=1}^K\left|\frac{e^{i\theta}+\exp(\vc X_k^T\vc\beta+\tau(\vc\beta))}{1 + \exp(\vc X_k^T\vc\beta+\tau(\vc\beta))}\right|^R
		\\
		= & e^{Ru(0)}\prod_{k=1}^K\left(\frac{1+2\exp(\vc X_k^T\vc\beta+\tau(\vc\beta))\cos\theta + \exp(2\vc X_k^T\vc\beta+2\tau(\vc\beta))}{1 + 2\exp(\vc X_k^T\vc\beta+\tau(\vc\beta))+\exp(2\vc X_k^T\vc\beta+2\tau(\vc\beta))}\right)^{R/2}
		\\
		\le & e^{Ru(0)}\prod_{k=1}^K\left(\frac{1+2\exp(\vc X_k^T\vc\beta+\tau(\vc\beta))\cos\delta + \exp(2\vc X_k^T\vc\beta+2\tau(\vc\beta))}{1 + 2\exp(\vc X_k^T\vc\beta+\tau(\vc\beta))+\exp(2\vc X_k^T\vc\beta+2\tau(\vc\beta))}\right)^{R/2}.
		\end{align}
		Using the inequalities $\cos\delta \le 1- \delta^2/6$ and $\log (1+x)\le x$, we have
		\begin{align}
		|e^{Ru(\theta)}|\le & e^{Ru(0)}\prod_{k=1}^K\left|1 - \frac{\exp(\vc X_k^T\vc\beta+\tau(\vc\beta))\delta^2}{3(1 + 2\exp(\vc X_k^T\vc\beta+\tau(\vc\beta))+\exp(2\vc X_k^T\vc\beta+2\tau(\vc\beta)))}\right|^{R/2}
		\\
		= & e^{Ru(0)}\exp\left[\frac{R}{2}\sum_{k=1}^{K}\log\left(1 - \frac{\exp(\vc X_k^T\vc\beta+\tau(\vc\beta))\delta^2}{3(1 + 2\exp(\vc X_k^T\vc\beta+\tau(\vc\beta))+\exp(2\vc X_k^T\vc\beta+2\tau(\vc\beta)))}\right)\right]
		\\
		\le & e^{Ru(0)}\exp\left[-\frac{R}{2}\sum_{k=1}^{K}\frac{\exp(\vc X_k^T\vc\beta+\tau(\vc\beta))\delta^2}{3(1 + 2\exp(\vc X_k^T\vc\beta+\tau(\vc\beta))+\exp(2\vc X_k^T\vc\beta+2\tau(\vc\beta)))}\right]
		\\
		\le & e^{Ru(0)}\exp\left(-\frac{RK\delta^2}{12}\right),
		\end{align}
		hence
		\begin{align}
		\label{eq:I2result}
		I_2\le \frac{1}{2\pi}\int_{[-\pi,\pi]\backslash[-\delta,\delta]}|e^{Ru(\theta)}|~d\theta \le \frac{1}{2\pi}\int_{-\pi}^{\pi}|e^{Ru(\theta)}|~d\theta \le e^{Ru(0)} \exp\left(-\frac{RK\delta^2}{12}\right).
		\end{align}	
		
		Summarizing \eqref{eq:I1result} and \eqref{eq:I2result}, we conclude that for any $\delta\in(0,\pi/2)$,
		\begin{align}
		\label{eq:fRu0relation}
		& e^{-Ru(0)}f_{R}(\vc\beta) 
		\\
		=& e^{O_p(1)R\delta^3}
		\left[
		O_p(1)R^{-1/2}
		-O_p(1)R^{-1}\delta^{-1}\exp\left(-O_p(1)R\delta^2\right)
		\right]
		+ \exp\left(-\frac{RK\delta^2}{12}\right).
		\end{align}
		Specifically, setting $\delta = R^{-1/3}$ in the above relation, we conclude that $e^{-Ru(0)}f_{R}(\vc\beta) = O_p(1)R^{-1/2} + o_p(1)$. Consequently,
		\begin{align}
			\lim_{R\to\infty}\frac{1}{R}\log f_{R}(\vc\beta) - u(0) = \lim_{R\to\infty}\frac{1}{R}\log \left[e^{-Ru(0)}f_{R}(\vc\beta)\right] = \lim_{R\to\infty}\frac{1}{R}\log \left[O_p(1)R^{-1/2} + o_p(1)\right] = 0.
		\end{align}
	\end{proof}



\end{document}